\documentclass[a4paper,final]{amsart} %
 \usepackage{amsthm,amsmath,amssymb,amscd,amsxtra,enumitem,amsgen,concrete}
\usepackage[english]{babel}                
\usepackage{amsrefs}
 \setenumerate[2]{label=(\theenumii)}
 \usepackage[all]{xy}
\newtheorem{theorem}{Theorem}[section]

\newtheorem{lemma}[theorem]{Lemma}
\newtheorem{e-proposition}[theorem]{Proposition}

\newtheorem{e-definition}[theorem]{Definition\rm}
\newtheorem{remark}{\bf Remark\/}
\theoremstyle{plain}

\newcommand{\nc}{\newcommand}
\newcommand{\cal}{\mathcal}
\newcommand{\on}{\operatorname}
\nc{\Spec}{\on{Spec}}
\nc \Zb{{\ensuremath{\bold Z}}} 
\nc \Oc{\ensuremath{\cal O}}
\nc \Ic{{\ensuremath{\cal I}}} 
\nc \Pc{{\ensuremath{\cal P}}}
\nc \Dc{\ensuremath{\mathcal D}}
\nc \Cb{{\ensuremath{\mathbf C}}} 
\nc \Ab{{\ensuremath{\mathbf A}}}
\nc \supp{\on {supp}}
\nc \Char{\on {Char}}
\nc \Ch{\on {Ch}}
\nc \Ann {\on {Ann }}
\nc \Mod{\on {Mod}}
\nc \End{\on {End}}
\nc \hto {\on {ht}}
\nc \Imo {\on {Im}}
\nc \mf{\mathfrak m} 
\nc \mh{\hat{\mf}} 
\nc \nf{\mathfrak n}
\nc \Of{\mathfrak O}
\nc \of{\mathfrak o}
\nc \rf{\mathfrak r}
\nc \tf{\mathfrak t}
\nc \mufr{{\mathbf \mu}}
\nc \hf{\mathfrak h} 
\nc \qf{\mathfrak q} 
\nc \bfr{\mathfrak b} 
\nc \kfr{\mathfrak k} 
\nc \pfr{\mathfrak p} 
\nc \af{\mathfrak a }
\nc \cf{\mathfrak c }
\nc \sfr{\mathfrak s} 
\nc \ufr{\mathfrak u} 
\nc \g{\mathfrak g} 
\nc \gA{\g_{\Ao}} 
\nc \lfr{\mathfrak l}
\nc \afr{\mathfrak a}
\nc \gfh{\hat {\mathfrak g}}
\nc \gl{\mathfrak { gl }}
\nc \Sl{\mathfrak {sl}}
\nc \SU{\mathfrak {SU}}
\nc{\Homf}{\mathfrak{Hom}}
\nc{\Theorem}[1]{Theorem~{#1}}
\nc{\Lemma}[1]{Lemma~{#1}} 
\nc{\Lem}[1]{({\sl Lem.} ~{#1})} 

\begin{document}

\selectlanguage{english}
                 \title{Two extensions of Hilbert's finiteness theorem}\label{invariantrings}

               \author{Rolf K{{\"a}l}lstr{\"o}m} \address{Department of Mathematics,University of
               G\"avle, 100 78 G\"avle, Sweden} \email{rkm@hig.se}  \date{\today}
\begin{abstract}

Let $S^{\cdot }$ be a noetherian graded algebra over a commutative $k$-algebra $A$, where $k$ is a
commutative ring, and assume it is a module over a Lie algebroid $\g_{A/k}$. If $S^\cdot$ is semi-simple over
$\g_{{A/k}}$ we prove that its  ring of invariants  $\bar S^{\cdot}$ is notherian. When $\g_{A/k}$ is a solvable Lie algebra
over $A$ we construct noetherian subalgebras of $S^\cdot$ from subsets of characters of
$\g_{A/k}$. We give similar results for noetherian modules over the pair $(S^{\cdot}, \g_{{A/k}})$.
\end{abstract}

\maketitle

               \section{}
               Let $A/k$ be a commutative $k$-algebra over the ring $k$, and $T_{A/k}$ its
               $A$-module of $k$-linear derivations, which is also a $k$-Lie algebra. Let $\g_{A/k}$
               be an $A$-module of finite type provided with a $k$-linear Lie bracket $[\cdot ,
               \cdot ]: \g_{{A/k}}\otimes_{k}\g_{{A/k}} \to \g_{{A/k}}$ and a homomorphism $\alpha :
               \g_{{A/k}} \to T_{{A/k}}$ of Lie algebras and $A$-modules, where we require
               $[\delta,a\eta ]= \alpha(\delta)(a)\eta + a[\delta, \eta]$, $a\in A$, $\delta,
               \eta \in \g_{A/k}$. We call the object $(A, \g_{{A/k}}, \alpha)$ a Lie algebroid,
               thinking of it as an infinitesimal Lie groupoid; the names Lie-Rinehart algebra and
               Atiyah algebra also circulate in the literature. Here are two basic families of
               examples of Lie algebroids: (i) $\g_{A/k} = T_{{A/k}}(I)\subset T_{{A/k}}$, the Lie
               subalgebroid of derivations $\partial$ which preserve an ideal $I\subset A$,
               $\partial (I)\subset I$; (ii) $\g_{A}= A\otimes_{k }\g_{k}$, where $\g_{k}$ is a
               $k$-Lie algebra provided with a homomorphism of $k$-Lie algebras $\phi: \g_{k}\to
               T_{{A/k}}$, where $[a_1\otimes \delta_{1}, a_{2}\otimes \delta_{2}] =
               a_1\phi(\delta_1)(a_{2})\otimes \delta_{2} -a_2\phi(\delta_2)(a_1)\otimes \delta_1 +
               a_1a_2\otimes [\delta_1, \delta_2] $, $a_1,a_2\in A, \delta_1, \delta_2\in \g_{k}$,
               and $\alpha : A\otimes_k\g_{{k}}\to T_{{A/k}}$, $\alpha (a\otimes \delta) = a
               \phi(\delta)$.
               
  A $\g_{A/k}$-module is an $A$-module $M$ and a homomorphism of $k$-Lie algebras $f: \g_{{A/k}}\to
  \End_k(M)$ such that $f(a\delta)(m)= af(\delta)(m)$ and $f(\delta)(am)= \alpha(a)m+ a
  f(\delta)(m)$, $\delta \in \g_{{A/k}}$, $a\in A$, $m\in M$. By a  graded $\g_{A/k}$-algebra we mean a graded commutative $A$-algebra
  $S^\cdot = \oplus_{i\geq 0} S^i$, which at the same time is a $\g_{A/k}$-module by a homomorphism of
  $A$-modules and Lie algebras $\phi : \g_{A/k} \to T_{S^\cdot/k}$, such that $\phi(\delta) (
  S^i)\subset S^i$, $\delta \in \g_{A/k}$.

  A graded $(S^\cdot , \g_{A/k})$-module is a graded $S^\cdot $-module
  and $\g_{A/k}$-module $M^\cdot= \oplus_{i\in \Zb} M^i$ such that
  $\delta \cdot M^i \subset M^i$ and $\delta (s m)= \delta (s) m + s
  \delta m$, $s\in S^\cdot, m\in M^\cdot $, $\delta \in
  \g_{A}$. We let $\Mod(S^\cdot, \g_{A/k})$ be the the category of
  graded $(S^\cdot , \g_{A/k})$- modules $M^\cdot$ which are of finite
  type over $S^\cdot$.

  When $M^\cdot \in \Mod(S^\cdot,\g_{A/k})$ we denote its invariant
  space $\bar M^\cdot = (M^\cdot)^{\g_{A}}= \{m\in M^\cdot \
  \vert \ \delta \cdot m = 0, \delta \in \g_{A}\}$. Clearly, $\bar
  S^\cdot= \oplus \bar S^i $ is a graded subring of $S^\cdot$, and
  $\bar M^\cdot$ is an $\bar S^\cdot$-module.

\section{}
The following result generalizes Hilbert's theorem about the finite
generation of invariant rings with respect to semi-simple Lie algebras
\cite{hilbert-invarianttheory}.
 
\begin{theorem} \label{thm001} Assume that $M^\cdot \in
  \Mod(S^\cdot, \g_{A}) $ where $S^\cdot$ is a graded noetherian
  $\g_{A/k}$-algebra.
  \begin{enumerate}
  \item Assume that $S^\cdot$ is semi-simple over $\g_{A}$ and $\bar
    S^0$ is noetherian.  Then $\bar S^\cdot$ is a graded and
    noetherian subring of $S^\cdot$.
  \item Assume (1) and also that $M^\cdot$ is semi-simple over
    $\g_{A}$. Then $\bar M^\cdot $ is of finite type over $\bar
    S^\cdot$.
  \end{enumerate}
\end{theorem}
\vskip 0.5\baselineskip
\begin{remark}\rm
  It is likely that the assumption in \Theorem{\ref{thm001}}, (2), can
 be relaxed so that (1) need not be assumed, at least when $A$ is a
  field.  The problem is that we need to know that the $\g_{A/k}$-module
  $Hom_A(M^\cdot, M^\cdot)$ is semi-simple, since then $S^\cdot$
  can replaced by its image in $Hom_A(M^\cdot, M^\cdot)$. However,
  it seems that in general it is a non-trivial problem to see when the
  $\g_{A/k}$-module of $A$-linear homomorphisms of semi-simple modules is
  semi-simple.
\end{remark}
\vskip 0.5\baselineskip

\begin{lemma}\label{extra-lemma}
If $N$ is an $(S^0, \g_{A/k})$-module of finite type over $S^0$
  and semi-simple over $\g_{A/k}$, then $\bar N$ is of finite type over
  $\bar S^0$.
\end{lemma}
\begin{proof}
  The $S^0$-submodule $S^0\bar N\subset N$ is of finite type over $S^0$, hence there exists an
  integer $ n$ and a surjective homomorphism of $(S^0, \g_{A/k})$-modules $\oplus_{i=1}^n S^0\to S^0
  \bar N$.  Since $S^0$ and $N$ are semi-simple, and therefore also $S^0 \bar N$ is semi-simple, it
  follows that this homomorphism is split, so applying the functor $Hom_{\g_{A/k}}(A, \cdot )$ the
  induced map $\oplus_{i=1}^n \bar S^0\to ( S^0 \bar N)^{\g_{A/k}}= \bar N$ is also surjective.
\end{proof}
\begin{lemma}\label{hilblemma} Make the assumptions in
  \Theorem{\ref{thm001}}, (1) and (2), so in particular $S^\cdot = A
  \bar S^\cdot \oplus Q$, where $Q$ is a semi-simple
  $\g_{A/k}$-submodule of $S^\cdot$, and $M^\cdot = A\bar M^\cdot
  \oplus M_1$.  Then $Q \bar M^\cdot \subset M_1$.
\end{lemma}
\begin{proof} Let $\Dc= \Dc(\g_{A/k})$ be the enveloping ring of differential operators of
  $\g_{{A/k}}$. For elements $q\in Q$ and $\bar m \in \bar M^\cdot$ we have $q\bar m= m_0 + m_1$, where $
  m_0 \in A \bar M^\cdot$ and $m_1 \in M_1$, and our goal is to prove $m_0=0$.  Since $Q$ is
  semi-simple it suffices to prove this when $\Dc\cdot q$ is simple.  Since $\Dc \cdot m_0 \subset A
  \bar M^\cdot$ it follows that $\Dc m_0 = \Dc \cdot (\Dc m_0)^{\g_{A/k}}$. Therefore, since $(\Dc
  \cdot q) \cap A \bar S^\cdot =\{0\}$, it follows that $Hom_{\Dc}(\Dc m_0, \Dc q)=0$; hence
  $\Ann_{\Dc}( m_0 )\not\subset \Ann_{\Dc}(q)$, so there exists an element $P$ in $\Dc$ such that $P
  q\neq 0 $ and $P m_0 =0$.  We have now
  \begin{displaymath}
    \Dc \cdot ( m_0 + m_1)= \Dc (q \bar m) = (\Dc\cdot q) \bar m =
    (\Dc \cdot 
    P\cdot q) \bar m = \Dc P (m_0 + m_1) = \Dc P m_1. 
  \end{displaymath}
  Therefore $m_0 \in M_1 \cap A \bar M^\cdot =\{0\}$.
\end{proof}
{\bf Proof of \Theorem{\ref{thm001}}}
(1): Since $S^\cdot$ is graded noetherian there exists an integer
$r$ such that the graded $\bar S^0$-submodule of $V = \oplus_{i=1}^r
\bar S^i\subset \bar S^\cdot_+$ generates $S^\cdot \bar
S_+^\cdot$ over $S^\cdot$. Let $B^\cdot$ be the subalgebra of
$\bar S^\cdot$ that is generated by $V$ and $\bar S^0$, so in
particular $B^i = \bar S^i$ when $0\leq i \leq r$.  Let $d > r$ be an
integer and assume by induction that $\bar S^i = B^i$ when $i < d$.
We have (as detailed below)
  \begin{displaymath} \bar S^d =  (\sum_{1\leq i \leq r} S^i\cdot
    \bar S^{ d-i}) \cap \bar S^d = \sum_{1\leq i \leq r} \bar S^i
    \cdot \bar S^{d-i} = \sum_{1\leq i \leq r}B^i \cdot B^{d-i} =
    B^d.
\end{displaymath}
The first equality follows from the inclusion $\bar S^d \subset
S^\cdot \bar S_+ = S^\cdot \cdot V $, noting that $V$ is
concentrated in degrees $1$ to $r$.  To see clearly the second
equality, we first have by \Lemma{\ref{hilblemma}} that $\bar
S^{d-i}Q^ i \subset Q^d$; then since $S^d =A \bar S^d\oplus Q^d $, if
$f\in Q^d $, $g\in A \bar S^d$, and $f+g\in \bar S^d$, it follows that
$f=0$.  The last equality follows by induction.  This proves that
$B^\cdot= \bar S^\cdot$. Since $B^0 = \bar S^0$ is noetherian and
$ V$ is of finite type over $B^0$ by \Lemma{\ref{extra-lemma}}, Hilbert's basis theorem
implies that $\bar S^\cdot$ is finitely generated.

(2): The proof is similar to (1). Since $M^\cdot$ is noetherian
there exists an integer $r$ such that the graded $\bar S^0$-submodule
$W= \oplus_{1\leq i \leq r} M^i$ of $\bar M^\cdot$ generates the
graded $S^\cdot$-module $S^\cdot \bar M_{>0}^\cdot\subset
M^\cdot$. There exists also an integer $t$ such that $M^i=0$ when $i
< t$.  Let $N^\cdot$ be the $\bar S^\cdot$-submodule of $\bar
M^\cdot$ that is generated by $\bar M_{\leq 0}^\cdot$ and $W$. Then $N^i= \bar M^i$ when $i\leq r$.
Let $d > r$ be an integer and assume that $N^i = \bar M^i$ when $i <
d$. We have then again by induction
\begin{displaymath} \bar M^d = (\sum_{1\leq i \leq r} S^i \cdot \bar
  M^{d-i}) \cap \bar M^d = \sum_{1\leq i \leq r} \bar S^i \cdot \bar
  M^{d-i} = \sum _{1\leq i \leq r} \bar S^i \cdot N^{d-i}= N^d.
\end{displaymath}  The second equality follows from the inclusion $Q^{d-i}
\cdot \bar M^i \subset M_1$ \Lem{\ref{hilblemma}}.  Therefore $\bar M^\cdot = N^\cdot$, and since $N^\cdot$
is finitely generated over $\bar S^\cdot$
by \Lemma{\ref{extra-lemma}}  this completes the proof.

\section{}
Now assume that $A=k$ and $\alpha =0$, so that $\g_k$ is a Lie algebra over $k$, and assume moreover
that $\g_k$ is solvable, so $S^{\cdot}_{k}$ need not be semi-simple even when $k$ is a field. In
fact, by a famous counter-example to Hilbert's 14th problem due to Nagata
\cite{nagata:14th},  there exists a noetherian $\g_k$-algebra $S^{\cdot}_{k}$ where $\g_k$ is a
solvable finite-dimensional Lie algebra over a field $k$, such that the invariant ring $\bar
S^\cdot$ is not noetherian.  
               
Instead of the invariant subring  we will make a construction of noetherian subalgbras of
$S^{\cdot}_{k}$ using  subsets of  characters of the  $\g_{k}$-module $S^{\cdot}_{k}$.
Let $\Ch(\g_k)= (\g_k/[\g_k,\g_k])^*$ be the character group. If $\chi \in
\Ch(\g_k) $, we put $S^\cdot_{\chi} = \{s\in S^\cdot \ \vert \ (X-
\chi (X))^n s =0, n\gg 1, X\in \g_k\}$. Similarly, we define
$M^\cdot_\chi$ when $M^\cdot$ is any graded $(\g_k,
S^\cdot)$-module, and put $\supp_{\g_k} M^\cdot= \{\chi \in
\Ch(\g_k)\ \vert\ M^\cdot_{{\chi}}\neq 0\}$.
\begin{theorem}\label{solv-theo}
  Let $S^\cdot = \oplus_{n\geq 0} S^n$ be a noetherian graded
  $\g_k$- algebra, where $\g_k$ is a solvable Lie algebra and put 
  $C=\supp_{\g_k} S^\cdot$. Then 
  \begin{displaymath}
    S^\cdot = \bigoplus_{\chi \in C  } S^\cdot_{\chi},
  \end{displaymath}
  and $C$ is a commutative sub-semigroup of $\Ch(\g_k)$ where
  the binary operation is induced by the ring structure of
  $S^\cdot$.
  \begin{enumerate}
  \item Let $\Gamma $ be a subsemigroup of $C$, and put $\Gamma^c =
    C\setminus \Gamma$.  Then
    \begin{displaymath}
      S^\cdot_\Gamma = \bigoplus_{\chi \in \Gamma} S^\cdot_{\chi}
    \end{displaymath}
    is a graded $\g_k$-algebra, and if moreover $\Gamma + \Gamma^c \subset
    \Gamma^c$, then $S^\cdot_\Gamma$ is noetherian.
  \item Let $M^\cdot$ be a finitely-generated graded $(\g_k,
    S^\cdot)$-module and put $C_M = \supp_{\g_k} M^\cdot$. Then
  \begin{displaymath}
    M^\cdot = \bigoplus_{\phi \in C_M} M^\cdot_\phi,
  \end{displaymath}
  and $C$ acts on $C_M$ in a natural way. Let $\Gamma $ be a
  subsemigroup of $C$, $\Phi$ be a subset of $C_M$, put $\Phi^c =
  C_M \setminus \Phi$, and consider the conditions:
  \begin{enumerate}
  \item $\Gamma \cdot \Phi \subset \Phi$,
  \item $ \Gamma^c \cdot \Phi \subset \Phi^c$.
  \end{enumerate}
  Then $(a)$ implies that
\begin{displaymath}
  M_\Phi^\cdot = \bigoplus_{\phi \in \Phi } M^\cdot_\phi
\end{displaymath}
is a graded $(S^\cdot_\Gamma, \g_k)$-module. If also (b) is
satisfied, then $M^\cdot_\Phi$ is of finite type over
$S_\Gamma^\cdot$.
\end{enumerate}
\end{theorem}
\begin{remark}\rm
  The assumption are trivially true when $\Gamma = \{0\}\subset C$, and $S^{\cdot}_{{(0)}} $ is the
  subalgebra on which $\g_{k}$ acts locally nilpotently. If moreover $k$ is a field of characteristic $0$
  and $S^{\cdot}_{{(0)}} $ is the symmetric algebra over a finite dimensional representation of
  $\g_{k}$, Weitzenb\"ock's theorem \cite{weitzenbock}  implies that its  invariant ring
  $\bar S^{\cdot}_{{(0)}}$ is noetherian.
\end{remark}
\begin{proof}
  (1): Since $\g_k$ is solvable it follows that
  $S^\cdot_{\chi}S^\cdot_{\chi'}\subset S^\cdot_{{\chi+
      \chi'}}$, implying that $C$ is a sub-semigroup of $\Ch(\g_k)$,
  and $S^\cdot_{\Gamma}$ is a $\g_k$-subalgebra. It remains to prove
  that $S^\cdot_\Gamma$ is noetherian.  Define a noetherian algebra
  $B^\cdot$ as in the proof of \Theorem{\ref{thm001}}, so $B^i =
  S_\Gamma^i$ when $0\leq i \leq r$.  Let $d > r$ be an integer and
  assume that $B^i = S_\Gamma^i$ when $0 \leq i \leq d-1$.  Since
  $\Gamma + \Gamma^c \subset \Gamma^c$ it follows that
  $S^\cdot_{\Gamma^c} S^\cdot_{\Gamma} \subset
  S^\cdot_{\Gamma^c}$. Therefore $S_\Gamma^{d-i} S^i_{\Gamma^c}
  \subset S^d_{\Gamma^c}$, and we get as before
  \begin{displaymath}
    S^d_{\Gamma} =  (\sum_{1\leq i \leq d-1} S^i\cdot 
S_\Gamma^{ d-i}) \cap S_\Gamma^d = \sum_{1\leq i \leq d-1} S_\Gamma^i \cdot 
S_\Gamma^{d-i} = \sum_{1\leq i \leq d-1}B^i \cdot B^{d-i} = B^d.
  \end{displaymath}
  This proves by induction that $B^\cdot = S^\cdot_\Gamma$.

  (2): The action of $C$ on $C_M$ is induced by the $S^\cdot$-action
  on $M^\cdot$.  The proof that $M^\cdot_\Phi$ is of finite type
  over $S^\cdot_\Gamma$ is analogous to the proof of
  \Theorem{\ref{thm001}}, (2).
\end{proof}

\end{document}